\let\ORIlabel\label
\let\ORIrefstepcounter\refstepcounter
\AddToHook{package/hyperref/before}{%
  \let\label\ORIlabel
  \let\refstepcounter\ORIrefstepcounter
}
\documentclass{siamart1116}

\usepackage[T1]{fontenc}
\usepackage{lmodern}
\usepackage{mathtools,amssymb}

\hypersetup{hidelinks}

\let\le\leqslant
\let\ge\geqslant

\newsiamremark{remark}{Remark}
\crefname{remark}{Remark}{Remarks}
\Crefname{remark}{Remark}{Remarks}

\newcommand{\R}{\mathbb{R}}

\newcommand{\abs}[1]{\left\lvert #1 \right\rvert}
\newcommand{\norm}[1]{\left\lVert #1 \right\rVert}

\DeclareMathOperator{\diag}{diag}

\DeclareMathOperator{\cond}{cond}

\newcommand{\U}{U}

\headers{Diagonal symmetrisation of tridiagonal Toeplitz matrices}{J.\ Verwee}

\title{\texorpdfstring{Diagonal symmetrisation\\of tridiagonal Toeplitz matrices}{Diagonal symmetrisation of tridiagonal Toeplitz matrices}%
\thanks{Submitted to the editors on \today.}}

\author{Johann Verwee%
\thanks{Independent Researcher (\email{mverwee@gmail.com}).}}

\begin{document}

\maketitle

\begin{abstract}
We develop a self-contained framework for real tridiagonal Toeplitz matrices\\$A_n(a,b,c)$ (diagonal $b$, subdiagonal $a$, superdiagonal $c$) in the symmetrisable regime $ac>0$.
A diagonal similarity transforms $A_n(a,b,c)$ into a symmetric Toeplitz matrix, yielding explicit eigenpairs, a Chebyshev determinant/characteristic polynomial formula, and a closed Green kernel for the inverse.
As an application we give sharp extremal eigenvalue and conditioning formulae in the natural weighted Hilbert space induced by this similarity.
Specialising to the classical repunit matrix $A_n(d,d+1,1)$, we show that $\det(A_n(d,d+1,1))=1+d+\cdots+d^{n}$ and obtain a finite cosine product factorisation of this repunit polynomial, together with quantitative bounds and an explicit inverse in terms of repunits.
\end{abstract}



\section{Introduction}\label{sec:intro}

Tridiagonal Toeplitz matrices are among the few structured matrix classes for which one can write down
closed formulae simultaneously for eigenvalues, determinants, and inverses.
The goal of this note is twofold.
First, we present a clean \emph{diagonal symmetrisation} principle which packages several classical computations
into a single statement and makes the natural weighted inner product transparent.
Second, we use this framework to isolate a repunit model as a concrete corollary.
Here a \emph{repunit} (in base $d$) refers to the polynomial $1+d+\cdots+d^{n}$, which also equals $(d^{n+1}-1)/(d-1)$ when $d\ne 1$.
We derive in that setting a finite cosine product factorisation together with information that goes beyond the determinant identity
(in particular, conditioning and explicit inverse kernels).

The central observation is that, when $ac>0$, the Toeplitz matrix $A_n\left( a,b,c \right)$ is diagonally similar
to a \emph{symmetric} tridiagonal Toeplitz matrix with off-diagonal entry $\sqrt{ac}$.
This similarity does not merely transport eigenvalues: it also identifies the correct Hilbert space
in which the original non-symmetric matrix is self-adjoint.
Once this is made explicit, several spectral and analytic features become immediate.

We now fix notation and state the main results.
Throughout, $n\ge 1$ is an integer, $a,b,c\in\R$ with $a\ne 0$ and $c\ne 0$, and we write
\begin{equation}\label{eq:sq}
s:=\sqrt{ac},\qquad q:=\frac{s}{c}.
\end{equation}
whenever $ac>0$.
Thus $s>0$ and $q\ne 0$, and one has $q^{2}=a/c$ and $cq=a/q=s$.

\section{\texorpdfstring{Diagonal symmetrisation and weighted\\self-adjointness}{Diagonal symmetrisation and weighted self-adjointness}}\label{sec:sym}

We begin with the diagonal similarity which drives the rest of the paper.
In the symmetrisable regime $ac>0$ the map is elementary, but it is useful to record it precisely
because it canonically produces a weighted inner product for which $A_n\left( a,b,c \right)$ is self-adjoint.

\begin{definition}[Tridiagonal Toeplitz model]\label{def:An}
For $n\ge 1$ and parameters $a,b,c\in\R$, we denote by $A_n\left( a,b,c \right)$ the $n\times n$ tridiagonal Toeplitz matrix
with diagonal entries $b$, subdiagonal entries $a$, and superdiagonal entries $c$, i.e.
\[
A_n\left( a,b,c \right)
=
\begin{pmatrix}
b & c & 0 & \cdots & 0\\
a & b & c & \ddots & \vdots\\
0 & a & b & \ddots & 0\\
\vdots & \ddots & \ddots & \ddots & c\\
0 & \cdots & 0 & a & b
\end{pmatrix}.
\]
\end{definition}

We now adopt the standing assumption $ac>0$ and we keep the parameters $s$ and $q$ from \eqref{eq:sq} fixed throughout.
For brevity we set
\[
\begin{aligned}
A_n &\coloneqq A_n\left( a,b,c \right), \qquad S \coloneqq S_n\left( b,s \right)=A_n\left( s,b,s \right),\\
D &\coloneqq \diag\left( 1,q,q^2,\dots,q^{n-1} \right), \qquad W \coloneqq D^{-2}.
\end{aligned}
\]

With this notation in place, the diagonal similarity becomes a one-line computation.

\begin{lemma}[Diagonal symmetrisation]\label{lem:symmetrisation}
One has
\begin{equation}\label{eq:similarity}
D^{-1} A_n D = S.
\end{equation}
Equivalently, $A_n$ is self-adjoint in the weighted inner product induced by $W$, in the sense that
\begin{equation}\label{eq:weighted}
A_n^{T} W = W A_n.
\end{equation}
\end{lemma}

\begin{proof}
Since $D$ is diagonal, the diagonal entries are unchanged.
For $1\le i\le n-1$ we have
\[
\left( D^{-1}A_nD \right)_{i,i+1}
= D_{ii}^{-1}\, c \, D_{i+1,i+1}
= q^{-(i-1)} c q^{i}
= cq
= s,
\]
and similarly
\[
\left( D^{-1}A_nD \right)_{i+1,i}
= D_{i+1,i+1}^{-1}\, a \, D_{ii}
= q^{-i} a q^{i-1}
= \frac{a}{q}
= s.
\]
This proves \eqref{eq:similarity}.
For \eqref{eq:weighted}, transpose \eqref{eq:similarity} to obtain
\[
D A_n^{T} D^{-1} = S_n\left( b,s \right)
= D^{-1}A_nD,
\]
and multiply by $D^{-1}$ on the left and by $D$ on the right.
Since $W=D^{-2}$, this yields $A_n^{T}W=WA_n$.
\end{proof}

\begin{remark}[The canonical weighted Hilbert space]\label{rem:weightedspace}
If $ac>0$ and $W=D^{-2}$ as above, then \eqref{eq:weighted} is equivalent to the statement that
$A_n$ is self-adjoint for the inner product
\[
\left\langle x,y \right\rangle_{W}:= x^{T} W y\qquad\left( x,y\in\R^{n} \right).
\]
In particular, $A_n$ has a complete set of eigenvectors that are orthogonal for
$\left\langle\cdot,\cdot\right\rangle_{W}$.
\end{remark}

The similarity of Lemma~\ref{lem:symmetrisation} reduces spectral questions to the symmetric case.
We next spell out the diagonalisation, the determinant in terms of Chebyshev polynomials, and a Green kernel for the inverse.

\section{Spectrum and Chebyshev determinant formula}\label{sec:spectrum}

The diagonalisation of symmetric tridiagonal Toeplitz matrices is classical and can be proven by a short
discrete Fourier--sine computation.
We include a complete proof because it is the most economical way to keep later arguments self-contained.

\begin{lemma}[Orthogonality for symmetric eigenvectors]\label{lem:symorth}
Let $S$ be a real symmetric matrix.
If $Su=\lambda u$ and $Sv=\mu v$ with $\lambda\ne \mu$, then $u^{T}v=0$.
\end{lemma}

\begin{proof}
We compute $\lambda u^{T}v=\left( Su \right)^{T}v=u^{T}S^{T}v=u^{T}Sv=\mu u^{T}v$.
Since $\lambda\ne \mu$, this forces $u^{T}v=0$.
\end{proof}

We now turn to explicit formulae for the spectrum and eigenvectors. The symmetrised picture is classical, and the diagonal similarity transports it back to $A_n$.

\begin{theorem}[Explicit eigenpairs]\label{thm:eigenpairs}
For $1\le k\le n$, define
\[
\theta_k:=\frac{k\pi}{n+1},
\qquad
\lambda_k:=b+2s\cos\left( \theta_k \right).
\]
Then $\lambda_1,\dots,\lambda_n$ are the eigenvalues of $A_n$.
Moreover, a corresponding right eigenvector $r^{(k)}\in\R^n$ may be chosen with entries
\begin{equation}\label{eq:rightvec}
r^{(k)}_j = q^{j-1}\sin\left( j\theta_k \right),\qquad 1\le j\le n,
\end{equation}
where $q$ is as in \eqref{eq:sq}.
\end{theorem}

\begin{proof}
By Lemma~\ref{lem:symmetrisation}, $A_n$ is similar to $S_n\left( b,s \right)$,
hence it suffices to diagonalise $S_n\left( b,s \right)$.
Fix $1\le k\le n$ and consider the vector $u^{(k)}\in\R^n$ with components
$u^{(k)}_j=\sin\left( j\theta_k \right)$.
For $2\le j\le n-1$ the $j$-th component of $S_n\left( b,s \right)u^{(k)}$ equals
\[
s\sin\left( \left( j-1 \right)\theta_k \right) + b\sin\left( j\theta_k \right) + s\sin\left( \left( j+1 \right)\theta_k \right).
\]
Using $\sin\left( \left( j-1 \right)\theta_k \right)+\sin\left( \left( j+1 \right)\theta_k \right)=2\cos\left( \theta_k \right)\sin\left( j\theta_k \right)$,
this becomes
\[
\left( b+2s\cos\left( \theta_k \right) \right)\sin\left( j\theta_k \right)=\lambda_k u^{(k)}_j.
\]
At $j=1$ and $j=n$ the same identity holds because
\[
\sin\left( 0\cdot \theta_k \right)=0
\qquad\text{and}\qquad
\sin\left( \left( n+1 \right)\theta_k \right)=\sin\left( k\pi \right)=0,
\]
which encode the boundary conditions.
Thus $S_n\left( b,s \right)u^{(k)}=\lambda_k u^{(k)}$.

The eigenvalues $\lambda_k=b+2s\cos\left( \theta_k \right)$ are pairwise distinct because
$\cos\left( \theta \right)$ is strictly decreasing on $\left( 0,\pi \right)$ and
$0<\theta_1<\cdots<\theta_n<\pi$.
Since $S_n\left( b,s \right)$ is symmetric, Lemma~\ref{lem:symorth} shows that the eigenvectors
$u^{(1)},\dots,u^{(n)}$ are mutually orthogonal in the Euclidean inner product; in particular they are linearly independent.
Therefore $S_n\left( b,s \right)$ admits $n$ linearly independent eigenvectors and
$\lambda_1,\dots,\lambda_n$ exhaust its spectrum.

Finally, Lemma~\ref{lem:symmetrisation} gives
\[
A_n\left( D u^{(k)} \right)
=
D S_n\left( b,s \right)u^{(k)}
=
\lambda_k \left( D u^{(k)} \right).
\]
Since $\left( Du^{(k)} \right)_j=q^{j-1}\sin\left( j\theta_k \right)$, this proves \eqref{eq:rightvec}.
\end{proof}

The eigenvalues yield a factorisation of the characteristic polynomial and, in particular, of the determinant.
It is convenient to express this determinant in terms of the Chebyshev polynomials $\U_m$ of the second kind,
characterised by $\U_0\left( x \right)=1$, $\U_1\left( x \right)=2x$, and the recursion
$\U_{m+1}\left( x \right)=2x\,\U_m\left( x \right)-\U_{m-1}\left( x \right)$.

\begin{proposition}[Determinant and characteristic polynomial]\label{prop:det}
Let $x:=b/\left( 2s \right)$. Then
\begin{equation}\label{eq:detCheb}
\det\left( A_n \right)= s^{\,n}\U_n\left( x \right).
\end{equation}
Equivalently, the characteristic polynomial of $A_n$ is given by
\[
\chi_n\left( t \right)=\det\left( tI_n-A_n \right)= s^{\,n}\U_n\left( \frac{t-b}{2s} \right).
\]
\end{proposition}

\begin{proof}
By similarity, we have
\[
\det\left( A_n \right)=\det\left( S_n\left( b,s \right) \right).
\]

Let $\Delta_n:=\det\left( S_n\left( b,s \right) \right)$.
Expanding along the first row shows that $\Delta_n$ satisfies the continuant recursion
\[
\Delta_n = b\,\Delta_{n-1}-s^2\,\Delta_{n-2}\qquad\left( n\ge 2 \right),
\]
with initial values $\Delta_0=1$ and $\Delta_1=b$.
Writing $\Delta_n=s^{\,n}\delta_n$ gives
\[
\delta_n = \frac{b}{s}\delta_{n-1}-\delta_{n-2}=2x\,\delta_{n-1}-\delta_{n-2}
\qquad\left( n\ge 2 \right),
\]
with $\delta_0=1$ and $\delta_1=2x$.
Thus $\delta_n=\U_n\left( x \right)$ for all $n$, proving \eqref{eq:detCheb}.
The characteristic polynomial formula follows by replacing $b$ with $t-b$ and noting that $s$ is unchanged.
\end{proof}

\section{A closed inverse kernel and decay in the gapped regime}\label{sec:inverse}

Beyond eigenvalues and determinants, diagonal symmetrisation also gives a clean route to explicit inverses.
There are several well-known formulae in the literature for inverses of tridiagonal Toeplitz matrices
(see, for instance, da Fonseca--Petronilho~\cite{daFonsecaPetronilho2001} and Usmani~\cite{Usmani1994});
here we record a particularly concise representation in terms of Chebyshev polynomials,
and we explain how it immediately implies exponential off-diagonal decay in the uniformly positive definite regime $b>2\sqrt{ac}$.

\begin{theorem}[Inverse kernel]\label{thm:inversekernel}
Let $x:=b/\left( 2s \right)$.
If $\U_n\left( x \right)\ne 0$, then $A_n$ is invertible, and for $1\le i\le j\le n$ one has
\begin{equation}\label{eq:inverseCheb}
\left( A_n^{-1} \right)_{ij}
=
\frac{\left( -1 \right)^{i+j}}{s}\,
q^{\,i-j}\,
\frac{\U_{i-1}\left( x \right)\U_{n-j}\left( x \right)}{\U_n\left( x \right)}.
\end{equation}
For $1\le j\le i\le n$ one also has the equivalent formula
\[
\left( A_n^{-1} \right)_{ij}
=
\frac{\left( -1 \right)^{i+j}}{s}\,q^{\,i-j}\,\frac{\U_{j-1}\left( x \right)\U_{n-i}\left( x \right)}{\U_n\left( x \right)}.
\]
Equivalently, the inverse satisfies the weighted symmetry relation
\[
\left( A_n^{-1} \right)_{ij}
=
q^{2\left( i-j \right)}\left( A_n^{-1} \right)_{ji}.
\]
\end{theorem}

\begin{proof}
By Lemma~\ref{lem:symmetrisation}, we have $A_n=D S D^{-1}$ with $S=S_n\left( b,s \right)$, hence
\[
A_n^{-1}=D S^{-1} D^{-1}
\qquad\text{and}\qquad
\left( A_n^{-1} \right)_{ij}=q^{\,i-j}\left( S^{-1} \right)_{ij}.
\]
It therefore suffices to establish \eqref{eq:inverseCheb} in the symmetric case $A_n=S$ (so $q=1$).
Write $b=2sx$.
By Proposition~\ref{prop:det}, the assumption $\U_n\left( x \right)\ne 0$ implies $\det\left( S \right)\ne 0$.

We adopt the standard extension $\U_{-1}\left( x \right):=0$.
Define sequences $\left( f_k \right)_{k=0}^{n+1}$ and $\left( g_k \right)_{k=0}^{n+1}$ by
\[
f_k:=\left( -1 \right)^{k-1}\U_{k-1}\left( x \right),
\qquad
g_k:=\left( -1 \right)^{k-1}\U_{n-k}\left( x \right).
\]
Then $f_0=0$ and $g_{n+1}=0$, and for every $1\le k\le n$ the Chebyshev recursion gives the homogeneous relations
\begin{equation}\label{eq:homrec}
s f_{k-1}+b f_k+s f_{k+1}=0,
\qquad
s g_{k-1}+b g_k+s g_{k+1}=0.
\end{equation}
Indeed, since $b=2sx$, we compute
\[
s f_{k-1}+b f_k+s f_{k+1}
=
\left( -1 \right)^{k}s\left( \U_{k-2}\left( x \right)-2x\,\U_{k-1}\left( x \right)+\U_{k}\left( x \right) \right)=0,
\]
and similarly for $g$.

Next, for $0\le k\le n$ define the discrete Wronskian
\[
W_k:= f_k g_{k+1}-f_{k+1} g_k.
\]
Using \eqref{eq:homrec} twice and cancelling terms, one checks that $W_k$ is constant in $k$.
More explicitly, multiplying the first identity in \eqref{eq:homrec} by $g_{k+1}$ and the second by $f_{k+1}$ and subtracting yields
\[
s\left( f_k g_{k+1}-f_{k+1} g_k \right)
=
s\left( f_{k+1} g_{k+2}-f_{k+2} g_{k+1} \right),
\]
hence $W_k=W_{k+1}$.
Evaluating at $k=n$ gives
\[
W_k=W_n=f_n g_{n+1}-f_{n+1} g_n = -f_{n+1} g_n.
\]
Since $f_{n+1}=\left( -1 \right)^{n}\U_n\left( x \right)$ and $g_n=\left( -1 \right)^{n-1}\U_0\left( x \right)=\left( -1 \right)^{n-1}$,
we obtain $W_k=\U_n\left( x \right)$ for all $k$.

For $1\le i,j\le n$, define
\[
G_{ij}:=
\frac{1}{s\,\U_n\left( x \right)}
\begin{cases}
f_i g_j,& i\le j,\\
f_j g_i,& i>j.
\end{cases}
\]
Then $G$ is symmetric, and for $i\le j$ we have
\[
G_{ij}=\frac{1}{s\,\U_n\left( x \right)}\left( -1 \right)^{i+j}\U_{i-1}\left( x \right)\U_{n-j}\left( x \right),
\]
which is exactly \eqref{eq:inverseCheb} in the case $q=1$.

It remains to show that $SG=I_n$.
Fix $j$ and consider the $j$-th column $u^{(j)}=\left( G_{1j},\dots,G_{nj} \right)^{T}$.
For $k<j$ we have $u^{(j)}_k=\frac{g_j}{s\U_n\left( x \right)}f_k$, so \eqref{eq:homrec} implies $\left( Su^{(j)} \right)_k=0$.
For $k>j$ we have $u^{(j)}_k=\frac{f_j}{s\U_n\left( x \right)}g_k$, so again $\left( Su^{(j)} \right)_k=0$.
At $k=j$, we compute using the definitions and the Wronskian identity $W_j=\U_n\left( x \right)$ that
\begin{align*}
\left( Su^{(j)} \right)_j
&= sG_{j-1,j}+bG_{jj}+sG_{j+1,j}\\
&= \frac{1}{\U_n\left( x \right)}\left( f_{j-1} g_j+\frac{b}{s}f_j g_j+f_j g_{j+1} \right)\\
&= \frac{1}{\U_n\left( x \right)}\left( -f_{j+1} g_j+f_j g_{j+1} \right)\qquad\text{by \eqref{eq:homrec} at $k=j$}\\
&= \frac{W_j}{\U_n\left( x \right)}\\
&= 1.
\end{align*}
Therefore $Su^{(j)}=e_j$ for every $j$, where $e_j$ is the $j$-th standard basis vector of $\R^n$, hence $SG=I_n$ and $G=S^{-1}$.
The general case follows from $A_n^{-1}=D S^{-1} D^{-1}$.
\end{proof}

To make decay estimates transparent in the uniformly positive definite regime, we recall the standard hyperbolic parametrisation of $\U_m$.

\begin{lemma}[Hyperbolic evaluation of $\U_m$]\label{lem:Uhyper}
For every $m\ge 0$ and every $\gamma\in\R$ one has
\[
\U_m\left( \cosh\left( \gamma \right) \right)=\frac{\sinh\left( \left( m+1 \right)\gamma \right)}{\sinh\left( \gamma \right)}.
\]
\end{lemma}

\begin{proof}
Let $F_m\left( \gamma \right):=\sinh\left( \left( m+1 \right)\gamma \right)/\sinh\left( \gamma \right)$.
Then $F_0\left( \gamma \right)=1$ and $F_1\left( \gamma \right)=2\cosh\left( \gamma \right)$.
The addition formula gives, for $m\ge 1$,
\[
\sinh\left( \left( m+2 \right)\gamma \right)
=
2\cosh\left( \gamma \right)\sinh\left( \left( m+1 \right)\gamma \right)-\sinh\left( m\gamma \right).
\]
Dividing by $\sinh\left( \gamma \right)$ yields
$F_{m+1}\left( \gamma \right)=2\cosh\left( \gamma \right)F_m\left( \gamma \right)-F_{m-1}\left( \gamma \right)$.
Thus $F_m$ satisfies the defining recursion and initial values of $\U_m$ evaluated at $x=\cosh\left( \gamma \right)$, and the claim follows.
\end{proof}

\begin{remark}[A hyperbolic form when $x>1$]\label{rem:hyperbolic}
When $x>1$, there exists $\gamma>0$ such that $x=\cosh\left( \gamma \right)$.
By Lemma~\ref{lem:Uhyper}, one has $\U_m\left( \cosh\left( \gamma \right) \right)=\sinh\left( \left( m+1 \right)\gamma \right)/\sinh\left( \gamma \right)$
turns \eqref{eq:inverseCheb} into a Green kernel:
for $1\le i\le j\le n$,
\[
\left( S_n\left( b,s \right)^{-1} \right)_{ij}
=
\frac{\left( -1 \right)^{i+j}}{s}\,
\frac{\sinh\left( i\gamma \right)\sinh\left( \left( n+1-j \right)\gamma \right)}{\sinh\left( \left( n+1 \right)\gamma \right)\sinh\left( \gamma \right)}.
\]
In this regime one reads off exponential off-diagonal decay directly.
\end{remark}

\begin{corollary}[Exponential decay of the inverse]\label{cor:decay}
Assume that $x=b/\left( 2s \right)>1$ and set $\eta:=x+\sqrt{x^{2}-1}>1$.
Then for every $1\le i,j\le n$,
\[
\abs{\left( S^{-1} \right)_{ij}}
\le
\frac{2}{s}\,
\frac{1}{\eta-\eta^{-1}}\,
\eta^{-\abs{i-j}}.
\]
Consequently,
\[
\abs{\left( A_n^{-1} \right)_{ij}}
\le
\frac{2}{s}\,
\frac{1}{\eta-\eta^{-1}}\,
\abs{q}^{\,i-j}\eta^{-\abs{i-j}}.
\]
\end{corollary}

\begin{proof}
By Remark~\ref{rem:hyperbolic} we may write the symmetric inverse entries as ratios of hyperbolic sines.
Writing $\eta=e^{\gamma}$ (so that $x=\cosh\left( \gamma \right)$), we have the exact identity
\[
\sinh\left( t \right)=\frac{\eta^{t}-\eta^{-t}}{2}\qquad(t\ge 0).
\]
In particular $\sinh\left( t \right)\le \eta^{t}/2$ and $\sinh\left( t \right)\ge \left( \eta^{t}-\eta^{-t} \right)/2$ for $t\ge 0$,
and inserting these two bounds into the formula of Remark~\ref{rem:hyperbolic} yields the stated estimate for $S^{-1}$.
The bound for $A_n^{-1}$ then follows from $A_n^{-1}=D S^{-1} D^{-1}$.
\end{proof}

\section{Extremal eigenvalues and conditioning in natural weighted norm}\label{sec:conditioning}

The symmetrisation also provides a natural answer to the question ``in which norm is $A_n$ well-conditioned?''.
The key point is that, although $A_n$ need not be symmetric in the Euclidean structure,
it is self-adjoint for $\left\langle\cdot,\cdot\right\rangle_W$ by Remark~\ref{rem:weightedspace}.
In that weighted Hilbert space, spectral calculus is identical to the symmetric case.

\begin{definition}[Weighted operator norm and condition number]\label{def:cond}
Let $W$ be symmetric positive definite.
For a matrix $M\in\R^{n\times n}$, we write
\[
\norm{M}_{2,W}
:=
\sup_{x\ne 0}\frac{\norm{Mx}_{W}}{\norm{x}_{W}},
\qquad
\norm{x}_{W}:=\sqrt{\left\langle x,x \right\rangle_W},
\]
and we define the weighted condition number
\[
\cond_{2,W}\left( M \right):=\norm{M}_{2,W}\norm{M^{-1}}_{2,W},
\]
when $M$ is invertible.
\end{definition}

Since $A_n$ is self-adjoint in the $W$-inner product, its operator norm and conditioning in that geometry are controlled by the spectrum.
We therefore record the extremal eigenvalues explicitly.

\begin{proposition}[Extremal eigenvalues]\label{prop:extremal}
Let $\lambda_k$ be as in Theorem~\ref{thm:eigenpairs}. Then
\[
\lambda_{\max}=b+2s\cos\left( \frac{\pi}{n+1} \right),
\qquad
\lambda_{\min}=b-2s\cos\left( \frac{\pi}{n+1} \right).
\]
In particular, the symmetric Toeplitz matrix $S$ is positive definite if and only if
\[
b>2s\cos\left( \frac{\pi}{n+1} \right).
\]
\end{proposition}

\begin{proof}
The eigenvalues of Theorem~\ref{thm:eigenpairs} are strictly decreasing in $k$ because $\cos\left( \theta_k \right)$ is strictly decreasing on $\left( 0,\pi \right)$.
Thus the maximum occurs at $k=1$ and the minimum at $k=n$, and $\cos\left( \theta_n \right)=-\cos\left( \theta_1 \right)$ gives the formula.
\end{proof}

Combining Proposition~\ref{prop:extremal} with the weighted self-adjointness of Lemma~\ref{lem:symmetrisation} yields a sharp condition number formula.

\begin{corollary}\label{cor:weightedcond}
(Sharp weighted conditioning.) Assume that $S$ is positive definite.
Then
\[
\cond_{2,W}\left( A_n \right)
=
\cond_{2}\left( S \right)
=
\frac{b+2s\cos\left( \frac{\pi}{n+1} \right)}{b-2s\cos\left( \frac{\pi}{n+1} \right)}.
\]
\end{corollary}

\begin{proof}
Since $A_n$ is self-adjoint in $\left\langle\cdot,\cdot\right\rangle_W$ and diagonalizable with real spectrum,
the weighted operator norm equals the spectral radius in absolute value, and similarly for $A_n^{-1}$.
Thus $\cond_{2,W}\left( A_n \right)=\abs{\lambda_{\max}}/\abs{\lambda_{\min}}$ in the positive definite regime.
By similarity, the spectrum of $A_n$ coincides with that of $S$, and Proposition~\ref{prop:extremal} gives the explicit ratio.
\end{proof}

In particular, if $x=b/\left( 2s \right)$ stays bounded away from $\cos\left( \pi/\left( n+1 \right) \right)$, then $\cond_{2,W}(A_n)$ remains uniformly bounded; it blows up as $b$ decreases to the critical value $2s\cos\left( \pi/(n+1) \right)$ from above.

\section{The repunit model and a finite cosine product}\label{sec:repunits}

We now specialise the general framework to the Toeplitz model whose determinant is a repunit.
This complements the cosine product factorisation by providing additional information beyond the determinant identity alone, such as a sharp weighted condition number and explicit inverse kernels.

\begin{definition}[Repunits]\label{def:repunit}
For $d>0$ and $m\ge 1$, the repunit of length $m$ in base $d$ is
\[
R_m\left( d \right):=\sum_{j=0}^{m-1}d^{j}.
\]
For $d\ne 1$ this equals $\left( d^{m}-1 \right)/\left( d-1 \right)$, and we adopt the continuous extension $R_m\left( 1 \right):=m$.
\end{definition}

We next introduce the Toeplitz matrix whose determinant produces these repunits.

\begin{definition}[Repunit Toeplitz matrix]\label{def:Vn}
For $d>0$ and $n\ge 1$, set
\[
V_n\left( d \right):=A_n\left( d,d+1,1 \right).
\]
Equivalently, $V_n\left( d \right)$ has diagonal entries $d+1$, superdiagonal entries $1$, and subdiagonal entries $d$.
\end{definition}

The following observation explains the terminology and connects the model to Proposition~\ref{prop:det}.

\begin{lemma}[Determinant equals a repunit]\label{lem:detrepunit}
For every $d>0$ and $n\ge 1$,
\[
\det\left( V_n\left( d \right) \right)=R_{n+1}\left( d \right).
\]
\end{lemma}

\begin{proof}
Let $\Delta_n:=\det\left( V_n\left( d \right) \right)$.
Expanding along the first row yields the recursion $\Delta_n=\left( d+1 \right)\Delta_{n-1}-d\,\Delta_{n-2}$ for $n\ge 2$,
with $\Delta_0=1$ and $\Delta_1=d+1$.
On the other hand, $R_{n+1}\left( d \right)$ satisfies the same recursion because
\[
R_{n+1}\left( d \right)=\left( d+1 \right)R_{n}\left( d \right)-dR_{n-1}\left( d \right)
\]
follows by comparing coefficients in $\sum_{j=0}^{n}d^{j}$.
Since the initial values agree, $\Delta_n=R_{n+1}\left( d \right)$ for all $n$.
\end{proof}

To make the spectral description concrete, note that for $V_n\left( d \right)$ we have $a=d$, $c=1$, hence $s=q=\sqrt{d}$.
In particular, the symmetrised matrix is $S_n\left( d+1,\sqrt{d} \right)$. Since $d+1 \geqslant 2\sqrt{d}$ (with equality only at $d=1$), this family lies on the boundary $x=1$ when $d=1$ and in the uniformly positive definite regime $x>1$ when $d\neq 1$.

\begin{theorem}[Repunit cosine product]\label{thm:repunitproduct}
For every $d>0$ and $n\ge 1$,
\begin{equation}\label{eq:repunitproduct}
R_{n+1}\left( d \right)
=
\prod_{k=1}^{n}\left( d+1+2\sqrt{d}\cos\left( \frac{k\pi}{n+1} \right) \right).
\end{equation}
\end{theorem}

\begin{proof}
By Theorem~\ref{thm:eigenpairs} applied to $A_n\left( d,d+1,1 \right)$, the eigenvalues of $V_n\left( d \right)$ are precisely the factors
in the product on the right-hand side of \eqref{eq:repunitproduct}.
Their product equals $\det\left( V_n\left( d \right) \right)$.
Conclude by Lemma~\ref{lem:detrepunit}.
\end{proof}

The product representation immediately yields quantitative corollaries.
We record two that are typical of what the symmetrisation framework gives ``for free''.

\begin{corollary}[Sharp weighted conditioning for the repunit matrix]\label{cor:repunitcond}
Let $d>0$, $n\ge 1$, and let $W=\diag\left( d^{-(j-1)} \right)_{1\le j\le n}$.
Then
\[
\cond_{2,W}\left( V_n\left( d \right) \right)
=
\frac{d+1+2\sqrt{d}\cos\left( \frac{\pi}{n+1} \right)}{d+1-2\sqrt{d}\cos\left( \frac{\pi}{n+1} \right)}.
\]
\end{corollary}

\begin{proof}
This is Corollary~\ref{cor:weightedcond} with $a=d$, $b=d+1$, and $c=1$.
The displayed $W$ equals $D^{-2}$ with $D=\diag\left( d^{(j-1)/2} \right)$.
\end{proof}

We finally rewrite the inverse kernel in terms of repunits, in a form that avoids Chebyshev notation.

\begin{corollary}[Inverse entries in terms of repunits]\label{cor:inverseRepunits}
Let $d>0$, $n\ge 1$, and $1\le i\le j\le n$.
Then
\[
\left( V_n\left( d \right)^{-1} \right)_{ij}
=
\frac{\left( -1 \right)^{i+j}}{\sqrt{d}}\,
d^{(i-j)/2}\,
\frac{R_i\left( d \right)R_{n+1-j}\left( d \right)}{R_{n+1}\left( d \right)}.
\]
\end{corollary}

\begin{proof}
Apply Theorem~\ref{thm:inversekernel} with $a=d$, $b=d+1$, $c=1$, $s=\sqrt{d}$ and $q=\sqrt{d}$.
It remains to note the Chebyshev identity
\[
\U_m\left( \frac{d+1}{2\sqrt{d}} \right)=d^{-m/2}R_{m+1}\left( d \right),
\]
which follows from the recursion for $\U_m$ and $R_{m+1}$ together with matching initial values at $m=0$ and $m=1$.
\end{proof}

The factorisation \eqref{eq:repunitproduct} identifies $R_{n+1}\left( d \right)$ as the determinant of a positive definite Toeplitz operator.
This immediately imports tools from numerical linear algebra.
For instance, Corollary~\ref{cor:repunitcond} gives an exact weighted condition number, and Corollary~\ref{cor:decay} yields exponential decay of
the inverse kernel, which implies rapid localisation of solutions of $V_n\left( d \right)x=f$ under localised forcing.

Diagonal symmetrisation is stable under finite-rank perturbations, so it can also be combined with
Sherman--Morrison--Woodbury-type identities to treat Toeplitz tridiagonal matrices whose four corners are modified.
For explicit eigenvalue and inverse formulae in that perturbed-corner setting, see Yueh--Cheng~\cite{YuehCheng}.

\end{document}